%% file: frobenius-de-rham.tex
\documentclass[10pt]{amsart}

\usepackage{amsfonts}

\usepackage{setspace} 
\input{header}

\title{The Dieudonn\'{e} modules and Ekedahl-Oort types of Jacobians of
hyperelliptic curves in odd characteristic}
\author{Sanath Devalapurkar}
\address{Department of Mathematics, Massachusetts Institute of Technology, 77
Massachusetts Avenue, Cambridge, MA 02142}
\email{\href{mailto:sanathd@mit.edu}{sanathd@mit.edu}}
\author{John Halliday}
\address{Department of Mathematics, University of Chicago, 5801 S Ellis Ave,
Chicago, IL 60637}
\email{\href{mailto:jhalliday@uchicago.edu}{jhalliday@uchicago.edu}}
\date{\today}
\begin{document}
\maketitle
\begin{abstract}
    Given a principally polarized abelian variety $A$ of dimension $g$ over an
    algebraically closed field $k$ of characteristic $p$, the $p$ torsion
    $A[p]$ is a finite flat $p$-torsion group scheme of rank $p^{2g}$. There
    are exactly $2^g$ possible group schemes that can occur as some such
    $A[p]$. In this paper, we study which group schemes can occur as $J[p]$,
    where $J$ is the Jacobian of a hyperelliptic curve defined over
    $\mathbb{F}_p$. We do this by computing explicit formulae for the action of
    Frobenius and its dual on the de Rham cohomology of a hyperelliptic curve
    with respect to a given basis. A theorem of Oda's in \cite{oda} allows us
    to relate these actions to the $p$-torsion structure of the Jacobian. Using
    these formulae and the computer algebra system Magma, we affirmatively
    resolve questions of Glass and Pries in \cite{glass-pries} on whether
    certain group schemes of rank $p^8$ and $p^{10}$ can occur as $J[p]$ of a
    hyperelliptic curve of genus $4$ and $5$ respectively.
\end{abstract}

\input{background}
\input{computation}
\input{applications}

\bibliographystyle{alpha}
\bibliography{main}
\end{document}

%% file: header.tex
\usepackage{hyperref}
\usepackage{mathtools}

\usepackage{todonotes}
\usepackage{multirow}
\usepackage{longtable}

\usepackage{amsmath}
\usepackage{amssymb}
\usepackage{float}
\usepackage{amsthm}
\usepackage{comment}

\usepackage{listings}
\usepackage{color}

\definecolor{dkgreen}{rgb}{0,0.6,0}
\definecolor{gray}{rgb}{0.5,0.5,0.5}
\definecolor{mauve}{rgb}{0.58,0,0.82}

\usepackage{color}
\usepackage[all]{xy}
\usepackage{verbatim}
\usepackage{setspace}
\usepackage[mathscr]{eucal}
\usepackage{graphicx}
\usepackage{tikz-cd}

\usepackage{amsmath}
\usepackage{url}
\usepackage{verbatim}

\usepackage{xy}
\xyoption{all}

\usepackage{amsthm}

\DeclareMathOperator{\Hom}{Hom}

\DeclareMathOperator{\spec}{Spec}

\DeclareMathOperator{\im}{im}
\DeclareMathOperator{\Jac}{Jac}

\newtheorem{lemma}{Lemma}[section]

\newtheorem{theorem}[lemma]{Theorem}

\newtheorem{prop}[lemma]{Proposition}

\newtheorem{abc}{abc}
\newtheorem{letterthm}[abc]{Theorem}

\theoremstyle{definition}

\newtheorem{question}{Question}

\newcommand{\FF}{\mathbb{F}}
\newcommand{\Z}{\mathbb{Z}}
\newcommand{\QQ}{\mathbb{Q}}

\newcommand{\wt}[1]{\widetilde{#1}}

\newcommand{\co}{\mathcal{O}}
\newcommand{\PP}{\mathbf{P}}

\newcommand{\GG}{\mathbf{G}}

\newcommand{\cA}{\mathcal{A}}

%% file: background.tex
\section{Introduction}

Let $k$ be a field of characteristic $p$. Recall that an elliptic curve $E$ over $k$ is a $1$-dimensional group variety over $k$. We denote the $p$-torsion of $E$ as a group scheme by $E[p]$, and the $p$-torsion points of $E$ over $\overline{k}$ by $E(\overline{k})[p]$. Although in characteristic $0$ the $p$-torsion of an elliptic curve over an algebraically closed field is isomorphic to $(\mathbb{Z}/p)^2$, in characteristic $p$ we can either have $E(\overline{k})[p] = \mathbb{Z}/p$ or $E(\overline{k})[p] = \{0\}$ as abstract groups. The former condition is called ordinary, and happens for a generic elliptic curve. The latter is called supersingular. Now considering the $p$-torsion as a group scheme and not merely a group of points, we have that if $E$ is ordinary then we have $E[p] = \Z/p\Z \oplus \mu_p$
(where $\Z/p\Z$ is the $p$-torsion of the $p$- divisible group $\QQ_p/\Z_p$,
and $\mu_p$ is the $p$-torsion of the multiplicative group scheme $\GG_m$). On
the other hand, if $E$ is a supersingular elliptic curve, the $p$-torsion
$E[p]$ sits in a nonsplit short exact sequence
$$0 \to \alpha_p \to E[p] \to \alpha_p \to 0,$$
where $\alpha_p \simeq \spec k[x]/x^p$ is the kernel of the Frobenius on the
additive group $\GG_a$.\\

\noindent For principally polarized abelian varieties of dimension greater than $1$, this generalizes in an especially rich way. Let $A$ be a principally
polarized abelian variety over $k$ of dimension $g$; this is a $k$-point of the
moduli space $\cA_g$ of principally polarized abelian varieties of dimension
$g$. Since $A$ is an abelian variety, it admits a multiplication-by-$p$ map
$[p]\colon A\to A$, which factors as the composite $V\circ F$. Here $F\colon A
\to A^{(p)}$ is the relative Frobenius, and $V$, called the Verschiebung, is
the dual of $F$. (There are identifications $\ker F = \im V$ and $\ker V = \im
F$, owing to the principal polarization on $A$.) Note that we also have $[p] =
F\circ V$. The $p$-torsion of $A$ is denoted $A[p]$.\\

\noindent The morphism $[p]$ is proper and flat of degree $p^{2g}$, so $A[p]$
is a $p$-torsion group scheme of rank $p^{2g}$, with induced morphisms $F$ and
$V$. The scheme structure of $A[p]$ is interesting, especially in dimensions
greater than $1$.

\noindent This motivates one to consider certain invariants of $A[p]$. The
\emph{$p$-rank} $f$ of $A$ is defined to be $\dim_{\FF_p} \Hom(\mu_p, A[p])$,
while the \emph{$a$-number} $a$ of $A$ is defined to be $\dim_{k}
\Hom(\alpha_p, A[p])$. It follows that the $p$-rank is the integer $f$ such
that $A[p](k) \simeq (\Z/p\Z)^f$, so that $p^f = \# A[p](k)$. This implies that
$$0\leq f\leq g.$$
Moreover, it is known that
$$0\leq a + f \leq g.$$

\noindent In analogy to the case of dimension $1$, we say that $A$ is
\emph{ordinary} if $f = g$. This implies that $a = 0$, and that $A[p] \simeq
(\Z/p\Z\oplus \mu_p)^f$. Most abelian varieties are ordinary, in that the
ordinary abelian varieties comprise an open subscheme of $\mathcal{A}_g$. At
the opposite end, we have \emph{superspecial} abelian varieties, which have
$p$-rank $0$ and $a$-number $g$. In this case, $A[p] = G^g$ where $G$ sits in the nonsplit exact sequence $$0 \to \alpha_p \to G \to \alpha_p \to 0.$$
\noindent In fact such an abelian variety is isogenous to $E^g$ where $E$ is a supersingular elliptic curve, so such abelian varieties are very rare.

\noindent In general, one can get a concrete algebraic handle on $A[p]$ by
describing its \emph{Dieudonn\'e module}. The scheme $A[p]$ is a
Barsotti--Tate group scheme (also known as a $\mathrm{BT}_1$ group scheme);
this means that it arises as the $p$-torsion of a $p$-divisible group over $k$.
While this is a rather large and complicated category to get a handle on, the
main theorem of Dieudonn\'e theory lets us linearize the problem.
\begin{theorem}\label{equiv}
	There is an equivalence of categories between $\mathrm{BT}_1$ group
	schemes and $k$-vector spaces with maps $F$ and $V$ such that $FV = VF
	= 0$, $\im F = \ker V$, and $\im V = \ker F$. This latter category is
	equivalent to finite (left) modules over the Cartier--Dieudonn\'e ring
	$$\mathrm{Cart}_p \coloneqq k[F,V]/(FV = VF = 0, Fx = x^p F, Vx^p =
	xV).$$
\end{theorem}

\noindent The \emph{Dieudonn\'e module} of an abelian variety $A$ over $k$ is
the image of $A[p]$ under the equivalence of categories in Theorem \ref{equiv};
this module will be denoted $D(A)$.

\noindent In this paper, we study which $D(A)$ can occur for $A$ the Jacobian
of a hyperelliptic curve $C$ over $\mathbb{F}_p$. To do this, we use the
following theorem of Oda to reduce the question to studying the actions of $F$
and $V$ on the de Rham cohomology of $C$.
\begin{theorem}\textnormal{\cite{oda}}
	Let $C$ be a smooth curve over $k$ of genus $g$. There is an
	isomorphism of $\mathrm{Cart}_p$-modules between $D(\Jac(C))$ and
	$H^1_{dR}(C)$.
\end{theorem}

\noindent Therefore, if we can effectively compute the action of $F$ and $V$ on
$H^1_{dR}(C)$, we can determine the group scheme structure of $\Jac(C)[p]$. One
is therefore interested in the following question\footnote{This is an
approximation to a harder question of computing the action of $F$ and $V$ on
the crystalline cohomology of $C$, which reduces mod $p$ to de Rham
cohomology.}:
\begin{question}
	What is the action of the Frobenius (and the Verschiebung) on
	$H^1_{dR}(C)$? Equivalently, what is the Dieudonn\'e module of the
	Jacobian of a hyperelliptic curve $C$ modulo $p$?
\end{question}
\noindent As a consequence of performing this computation, we will also be able
to describe the Ekedahl--Oort type (reviewed below) of Jacobians of
hyperelliptic curves.

\noindent In characteristic $2$, Elkin and Pries \cite{elkin-pries} have
computed the Dieudonn\'e module of the Jacobian of a hyperelliptic curve. We
will recall their result here.
\begin{theorem}\textnormal{\cite{elkin-pries}}
	Let $C\colon y^2 - y = f(x)$ be a hyperelliptic curve over $k$ (of
	characteristic $2$). Let $B$ denotes the set of branch points of $C\to
	\PP^1$, and write
	$$f(x) = \sum_{\alpha\in B} f_\alpha(x_\alpha),$$
	where $x_\alpha = (x-\alpha)^{-1}$ and $f_\alpha(x)\in xk[x^2]$ is a
	polynomial of certain degree with no monomials of even exponent. Let
	$Y_\alpha$ denote the curve $y^2 - y = f_\alpha(x)$. Then
	$$H^1_{dR}(C) \simeq D(\Jac(C)) \simeq (\Z/p\Z\oplus\mu_p)^{\#B -
	1}\oplus\bigoplus_{\alpha\in B}H^1_{dR}(Y_\alpha).$$
\end{theorem}
\noindent A similar result cannot be made in such generality in odd
characteristics, since we do not have Artin--Schreier properties.

\noindent The goal of this paper is to extend the above computation to odd
characteristics. Let $C$ be a hyperelliptic curve of genus $g$ defined by an
equation $y^2 = f(x)$ over a perfect field $k$ of characteristic $p\geq 3$.
Then there is a basis for the mod $p$ Dieudonn\'e module of $\Jac(C)$ in which
the matrices for the Frobenius and Verschiebung action can be explicitly
described in terms of the coefficients in the polynomial $f(x)$.

\noindent This computation can be used to study the behavior of certain special
types of hyperelliptic curves.
\begin{letterthm}
    \label{thm1}
    Let $d$ be an odd prime such that $g = \frac{d-1}{2}$ is also prime. Let
    $p$ be a prime not equal to $d$. Then the Jacobian of the hyperelliptic
    curve $y^2=x^d+1$ is ordinary if $p \equiv 1 \pmod d$, and has $p$-rank $0$
    if $p \not\equiv 1 \pmod d$.
\end{letterthm}
\newtheorem*{introthm1}{Theorem \ref{thm1}}
\noindent We can also answer some questions posed by Glass and Pries in
\cite{glass-pries} regarding the existence of certain hyperelliptic curves
whose Jacobians realize a certain Dieudonn\'e module:
\begin{letterthm}
    \label{thm2}
    Let $G$ denote the group scheme corresponding to the $p$-torsion of an
    abelian variety of dimension $3$ with $p$-rank $0$ and $a$-number $1$. For
    $p = 3, 5, 7$, there is a hyperelliptic curve $C$ of genus $4$ such that
    $$\Jac(C)[p] = G\oplus(\Z/p\oplus\mu_p).$$
    For $p=3,5$, there is a hyperelliptic curve $C$ of genus $5$ such that
    $$\Jac(C)[p] = G\oplus(\Z/p\oplus\mu_p)^2.$$
\end{letterthm}
\newtheorem*{introthm2}{Theorem \ref{thm2}}
\noindent Our main obstruction to extending these results to higher primes and
genus is twofold: first, the formulae obtained are inhumanely complicated,
even at small primes; second, the Magma code written is not optimal. The latter
problem should be something which can be easily fixed. This may allow for
extensions of the above results.

\subsection{Acknowledgements}
We would like to thank David Zureick-Brown for suggesting this project, and the
Emory REU for allowing us to conduct this research. We thank David
Zurieck-Brown and Jackson Morrow for comments on a draft. Thanks to many
people, Emory was an intellectually stimulating environment: we are
particularly grateful to Noam Kantor, Jackson Morrow, Ken Ono, and David
Zureick-Brown. We would also like to thank Rachel Pries and Sammy Luo for
helpful discussions regarding the topic of this paper. Lastly, we'd like to
thank the other participants in the Emory REU for making our summer fun and
enjoyable.

This research was supported by the Templeton World Charity Foundation and NSF
Grant DMS-1557690.

\section{Background}
\subsection{The Hasse invariant}
\subsubsection{Elliptic curves}
Let $E$ be an elliptic curve defined by $y^2 = f(x)$ over a field $k$ of
characteristic $p$. There are multiple ways to state that $E$ is ordinary, some
of which are as follows.
\begin{itemize}
    \item We have $E[p^k]\simeq \mathbf{Z}/p^k$ for all $k$.
    \item The trace of Frobenius on $E$ is nonzero.
    \item The Newton polygon of $E$ is ``as low as possible'', i.e., has a line
	of slope $-1$ from $(0,1)$ to $(1,0)$, and has a line of slope $0$ from
	$(1,0)$ to $(2,0)$.
    \item The formal group associated to $E$ is of height $1$, i.e., is
	isogenous to $\mathbf{G}_m$.
    \item Let $H_p$ denote the coefficient of $x^{p-1}$ in $f(x)^{(p-1)/2}$.
	This is known as the \emph{Hasse invariant}. Then $H_p\neq 0$.
\end{itemize}
We will be interested in generalizations to higher-dimensional abelian
varieties of the last characterization.\\


\noindent The Hasse invariant is easy to compute, but its significance is
opaque. To explain the definition, we need to set some notation. Let $F$ denote
the Frobenius on $E$; this is a morphism of schemes $E\to E^{(p)}$. The dual of
this is the Verschiebung, which is a morphism $V\colon E^{(p)}\to E$.\\

\noindent The Verschiebung induces a morphism on (co)tangent spaces, and in
particular, gives a morphism $V\colon\Omega^1_{K} \to \Omega^1_{K^p}$, where $K
= K(E)$ and $K^p$ is $K(E^{(p)})$. (Note that the Frobenius also gives a
morphism $\Omega^1_{K^p}\to \Omega^1_{K}$.) We can compute this map $V$
explicitly by fixing a nice basis for the differential forms of degree $1$ and
of the first kind (these are exactly the global sections of $\Omega^1_E$). This
is just a one-dimensional $k$-vector space, so we may choose the basis
consisting of the vector $dx/y \coloneqq \omega$.\\

\noindent To determine $V$, it suffices to determine $V(\omega)$. Explicitly,
we find that we can write any element of $\Omega^1_{K}$ in the form
$d\phi+\eta^p x^{p-1} dx$, where $\phi,\eta\in K$ and $\eta^p\in K^p$, and that
$$V(d\phi + \eta^p x^{p-1} dx) = \eta\ dx;$$
this is exactly as one would expect the Verschiebung to behave, since it is
dual to the Frobenius. Computing $V(\omega)$ will be a moment away if we can
write $\omega$ in this form. Define $c_i$ by the expansion
$$f(x)^{(p-1)/2} = \sum^{3(p-1)/2}_{j=0}c_j x^j dx.$$
Then (as in \cite[\S V.4]{silverman})
$$\omega = y^{-p} \left(y^2\right)^{(p-1)/2} dx = y^{-p}\sum^{3(p-1)/2}_{j=0}
c_j x^j dx.$$
\noindent One can then compute (we will do this in more generality below) that
$$V(\omega) = c^{1/p}_{p} dx/y.$$
Similarly, one finds that $F(\omega) = c_p d^px^p/y^p$, where $d^p \colon
K^p\to \Omega^1_{K^p}$ is the universal derivation.\\

\noindent In particular, the nonvanishing of the Hasse invariant is equivalent
to $V$ inducing a nonzero map on tangent spaces. This means that the
multiplication-by-$p$ map (which is the Verschiebung composed with the
Frobenius) will have seperable degree $p$ --- and one can show that this is
equivalent to $E[p]$ having $p$ connected components.

\subsubsection{Hyperelliptic curves}
The method of generalization to hyperelliptic curves is an exact analogue of
the above. We follow \cite{yui}. Let $C\colon y^2 = f(x)$ now be a
hyperelliptic curve of genus $g$, so that the degree of $f$ is $2g+1$ or
$2g+2$. Then -- exactly as above -- we can consider the effect of the
Verschiebung and the Frobenius on the sheaf of Kahler differentials.

This gives the ``modified Cartier'' and the ``Cartier'' operators,
respectively. We may pick the basis (as a $k$-vector space) for the global
sections of $\Omega^1_{C}$ given by
$$\omega_i \coloneqq x^{i-1} dx/y,$$
for $1\leq i\leq g$. Then, we can rewrite
$$\omega_i = y^{-p} x^{i-1} \sum^{(p-1)/2\cdot (2g+1)}_{j=0} c_j x^j dx,$$
and attempt to compute the effect of applying $V$ and $F$ to each of the basis
vectors.

\noindent Again, we find that we can write any element of $\Omega^1_{K}$ in the
form $d\phi+\eta^p x^{p-1} dx$, where $\phi,\eta\in K$ and $\eta^p\in K^p$, and
that $V$ sends this to $\eta \ dx$.

\begin{lemma}\label{vaction}
    The action of $F$ and $V$ on $H^0(C,\Omega^1_C)$ is determined by the
    following equations:
    \begin{align}
	V(\omega_i) & = \sum^{g-1}_{j=0}c^{1/p}_{(j+1)p - i} \frac{x^j}{y} dx,\\
	F(\omega_i) & = 0.
    \end{align}
\end{lemma}
\begin{proof}
    It is clear that $F(\omega_i) = 0$, so we will just compute $V(\omega_i)$.
    This computation will be immediate if we can write $\omega_i$ in the form
    $d\phi+\eta^p x^{p-1} dx$, which we can do: sinc
    $$\omega_i = y^{-p} \left(\sum^{(p-1)/2\cdot (2g+1)}_{j=0} c_j x^{j+i-1}
    \right)dx,$$
    we get
    \begin{align*}
	\omega_i & = y^{-p}\left(\sum_{i+j\neq 0\mod p}c_j x^{j+i-1}\right) +
	y^{-p} \left(\sum_{j}c_{(j+1)p-i}x^{(j+1)p-1}dx\right)\\
	& = d\left(y^{-p}\sum_{i+j\neq 0\mod p}\frac{c_j}{j+i} x^{j+i}\right) +
	\sum_{j}c_{(j+1)p-i}\frac{x^{jp}}{y^p}x^{p-1}dx.
    \end{align*}
    It follows that
    $$V(\omega_i) = \sum^{g-1}_{j=0}c^{1/p}_{(j+1)p-i}\frac{x^j}{y}dx,$$
    as desired.
\end{proof}
\noindent Define a matrix $A$ via $A_{i,j} \coloneqq c_{ip-j}$; then, Lemma
\ref{vaction} implies that if $A^{(1/p)}$ denotes the matrix such that
$A^{(1/p)}_{i,j} = \left(A_{i,j}\right)^{1/p}$, the Verschiebung $V$ acts via
$A^{(1/p)}$. This matrix is called the \emph{Hasse--Witt matrix}. It is clear
that the Hasse--Witt matrix of $C$ is unique up to transformations of the form
$A\mapsto S^{(p)} A S^{-1}$, where $S^{(p)}_{i,j} \coloneqq (S_{i,j})^p$.

As an example of how the Hasse--Witt matrix may be used to detect
supersingularity, we state the following theorem of Yui.
\begin{theorem}[\cite{yui}]
    Let $C$ be a hyperelliptic curve of genus $g$ over $k$. If the Hasse--Witt
    matrix of $C$ is $0$ in $k$, then the Jacobian $J(C)$ is supersingular and
    is isogenous to $g$ copies of a supersingular curve (over some finite
    extension of $k$).
\end{theorem}
\subsection{de Rham cohomology}
Notice that $F$ and $V$ both take exact forms to zero. (In fact $F$ takes all
forms to $0$.) In fact, the Hasse--Witt matrix $A$ is exactly the matrix (in the
basis written down above) for the action of the Frobenius on
$H^1(C,\mathcal{O}_C)$. Equivalently, by Serre duality, it is the matrix of the
Verschiebung on the zeroth cohomology of the de Rham complex.\\

\noindent There is a larger vector space which contains more information than
just $H^1(C,\mathcal{O}_C)$ and $H^0(C,\Omega^1_C)$. This is the (first) de
Rham cohomology of $C$. Recall that the Hodge--de Rham spectral sequence runs
$$E_1^{p,q} = H^p(X,\Omega^p_{X/K}) \Rightarrow H^{p+q}_{dR}(X).$$
The resulting filtration on $H^\ast_{dR}(X)$ is called the \emph{Hodge
filtration}.
\begin{theorem}
    The Hodge--de Rham spectral sequence collapses at the $E_1$-page if $X$ is
    smooth and proper.
\end{theorem}
\noindent The associated gradation of the Hodge filtration is exactly
$\mathrm{gr}^p H^i_{dR}(X) = H^{i-p}(X,\Omega^p_{X/K})$. Let $h^{p,q} = \dim_K
H^q(X,\Omega^p_{X/K})$; then Serre duality gives an equality $h^{p,q} =
h^{n-p,n-q}$, where $n$ is the dimension of $X$. One also has ``conjugation
symmetry'', which says that $h^{p,q} = h^{q,p}$.\\

\noindent We can, in particular, consider the Hodge filtration for
$H^1_{dR}(X)$. Given the description above, we have a short exact sequence
$$0\to H^0(X,\Omega^1_{X/R}) \simeq \mathrm{gr}^1 H^1_{dR}(X) \to H^1_{dR}(X)
\to \mathrm{gr}^0 H^1_{dR}(X) \simeq H^1(X,\mathcal{O}_X) \to 0.$$
The two associated gradations are of the same dimension.\\

\noindent As above, suppose that $X = C$ is a smooth geometrically connected
projective hyperelliptic curve over $k$ of genus $g$. It follows from GAGA that
$H^1_{dR}(C)$ is a $2g$-dimensional $k$-vector space. Since $C$ lives over $k$,
we still have an action of $F$ and $V$ on $H^1_{dR}(C)$. The rest of this paper
is devoted to finding explicit formulae for this action.

%% file: computation.tex
\section{The Dieudonn\'e module of Jacobians}
\subsection{A basis for the first de Rham cohomology group}
Let $C$ be a hyperelliptic curve defined by an equation $y^2 = f(x)$ over a
perfect field $k$ of odd characteristic $p$, where $f(x)$ is a separable
polynomial of degree $d$. By standard arguments (see \cite[\S 2.8]{poonen}) if
$d$ is odd we have that $d = 2g+1$ where $g$ is the genus, and if $d$ is even
we have that $d = 2g+2$. If we let $\mathbb{A}^1_x$ and $\mathbb{A}^1_u$ be the
standard affine charts on $\mathbb{P}^1$, where $u = 1/x$, we have that $C$ is
defined over $\mathbb{A}^1_x$ by the equation $y^2 = f(x)$, and over
$\mathbb{A}^1_u$ by the equation $v^2 = u^{2g+2}f(1/u)$, glued together by
setting $v = y/x^{g+1}$. Let $\pi\colon C \to \mathbb{P}^1$ be the projection.
Let $P_0=[0:1]$ and $P_\infty=[1:0]$.\\

\noindent On $C$,
$$(x) = \pi^*P_0 - \pi^*P_\infty.$$
If $d$ is odd, this means there will be a single point of $C$  above
$P_\infty$, so $x$ will have a pole of order $2$ there, and if $d$ is even,
$(x)$ will have a pole of order $1$ at each of the two $\bar{k}$ points above
$P_\infty$. Whether $d$ is even or odd, we know that
$$(dx/y) = (g-1)\pi^*P_\infty,$$
as in \cite[\S 2.8]{poonen}. If $d$ is odd, $y$ will have a pole of order
$d=2g+1$ above $P_\infty$, and if $d$ is even, $(y)$ above infinity will look
like $(g+1) \pi^*P_\infty$.\\

\noindent Recall that a canonical basis for $H^0(C,\Omega^1_C)$ is given by the
$1$-forms $\tau_{i-1} = x^{i-1} dx/y$ for $i$ ranging between $1$ and $g$. To
compute a basis for $H^1(C, \mathcal{O}_C)$, we use the two term Cech cover
$\mathscr{C}$ given by $U = \pi^{-1}(\mathbb{A}^1_u)$ and $V =
\pi^{-1}(\mathbb{A}^1_x)$. Moreover,
$$U \cap V = \spec{ k[x,y,x^{-1}]/(y^2-f(x))},$$
so $\Gamma(U \cap V, \mathcal{O}_{U \cap V})$ has a basis given by $\{x^i\}_{i
\in \mathbb{Z}} \cup \{yx^j\}_{j \in \mathbb{Z}}$.\\

Similarly,
$$\Gamma(V, \mathcal{O}_V)=k[x,y]/(y^2-f)$$
has a basis given by $x^i$ for all non-negative $i$ and $yx^j$ for all
nonnegative $j$. $\Gamma(U, \mathcal{O}_U)=k[x^{-1}, yx^{-(g+1)}]/(y^2 -f)$ has
a basis given by $x^i$ for all non-positive $i$ and $yx^j$ for all $j\leq
-(g+1)$. Thus the basis vectors of
$$H^1(C, \co_C)=\Gamma(U \cap V, \co_{U \cap V})/\Gamma(U, \mathcal{O}_U)\oplus
\Gamma(V, \mathcal{O}_V)$$
consist of $\{y/x^j\}_{j = 1, \dots, g}$.\\

\noindent We can now extend the bases of $H^0(C,\Omega^1_C)$ and $H^1(C,\co)$
to a basis of $H^1_{dR}(C, \mathscr{C})$. We begin by noting that the basis of
$H^0(C,\Omega^1_C)$ can be easily extended by defining
$$\wt{\tau}_i = \left(x^i\frac{dx}{y},x^i\frac{dx}{y},0\right),$$
as $i$ ranges from $0$ to $g-1$.

To extend the basis $H^1(C, \co_C)$ we will write out the polynomial for $f(x)$ as
$$f(x) = \sum^d_{k=0}a_kx^k.$$
We can then compute:
\begin{align*}
    d(y/x^j) & = \frac{dy}{x^j} - j \frac{ydx}{x^{j+1}}\\
    & = \frac{1}{x^j}\left(\frac{f'}{2y} - j\frac{y}{x}\right)dx\\
    & = \frac{1}{x^j} \left(f'(x) - 2jf(x)/x\right) \frac{dx}{2y}\\
    & = \sum_{k=0}^d (k-2j)a_k x^{k-j-1}\frac{dx}{2y}.
\end{align*}
Motivated by this, define differential forms (for $j$ ranging between $1$ and
$g$)
\begin{align*}
    \omega_1^{(j)} & = \sum_{k=0}^j (k-2j)a_k x^{k-j-1}\frac{dx}{2y},\\
    \omega_2^{(j)} & = -\left(\sum_{j+1}^d (k-2j)a_k
    x^{k-j-1}\frac{dx}{2y}\right).
\end{align*}
As $dx/y$ is regular and $\omega_1^{(j)}$ is the product of a polynomial in
$1/x$ and $\frac{dx}{y}$, $\omega^{(j)}_1$ has a pole only at $0$. A similar
argument proves that $\omega_2$ has a pole only at $\infty$. This discussion
implies the following result.

\begin{prop}
    The elements $(\omega^{(j)}_1,\omega^{(j)}_2,yx^{-j})=\wt{\eta}_i$ for
    $1\leq j\leq g$, along with the elements $\wt{\tau}_i$ for $0\leq i\leq
    g-1$, form a basis for $H^1_{dR}(C)$.
\end{prop}

\noindent Let $j\leq 0$. Then $y/x^j$ is a polynomial in $x$ and $y$, and the
above computation shows that $d(y/x^j)$ is regular everywhere on $V$, and has
poles only above $P_\infty$. Define
\begin{align*}
    \omega_1^{(j)} & = 0,\\
    \omega_2^{(j)} & = -d(y/x^j)= -\left(\sum_{k=0}^d (k-2j)a_k x^{k-j-1}\frac{dx}{2y}\right).
\end{align*}
For $j \geq g+1$, consider the differential form
$$\wt{\omega}^{(j)} = d(y/x^j) =  \sum_{k=0}^d (k-2j)a_k x^{k-j-1}\frac{dx}{2y}.$$
This is defined differently from the other differential forms $\omega_1^{(j)}$,
because the sum ranges from $1$ to $d$. This means that if $j<d$, there will be
some positive terms of $x$ in the sum. Since the largest power of $x$ is
$x^{d-j-1}$ (which has order $-2d+2+2j$ over $P_\infty$), and $dx/y$ has a zero
of order $2g-2$ over $P_\infty$ \footnote{In the case that $d$ is even and
there are two points over $P_\infty$, both $x$ and $dx/y$ will in fact have two
poles resp. zeroes of the same degree over each of the points over $P_\infty$,
so to prove a function is regular it suffices to prove that the total valuation
over $P_\infty$ is positive, which we denote by $(\cdot)_{P_\infty}$.} we have
$$\left(x^{d-j-1} \frac{dx}{y}\right)_{P_\infty} = 2g-2  - 2d + 2 + 2j = 2(g-d)
+ 2j = 2j - 2(d-g).$$
If $d$ is odd, $d-g = g+1$, and since $j \geq g+1$, this is regular at
$P_\infty$; hence $\wt{\omega}^{(j)}$ is a well-defined element of
$\Gamma(U,\Omega^1_C)$. If $d$ is even, $d-g = g+2$, so the only case where
this might have poles over $P_\infty$ is if $j = g+1$. But note that the
coefficient of $x^{d-j-1}$ is $d-2j$, so if $j = g+1$, $d-2j = 0$ and this term
vanishes. Thus in either case we have a well-defined element of
$\Omega^1_C(U)$.\\

\noindent Using the basis determined here, we can compute the action of
Verschiebung and Frobenius on $H^1_{dR}(C)$.

\subsection{The Verschiebung action}
To ease notation, let us redefine 
$$\wt{\tau}_i = \left(x^i\frac{dx}{y},x^i\frac{dx}{y},0\right),$$
as $i$ ranges from $1$ through $g$.\\

\noindent We kick off by noting that the Verschiebung is trivial on
$H^1(C,\mathcal{O}_C)$, simply because of Serre duality (the Frobenius is
trivial on $H^0(C,\Omega^1_C)$). Thus we only have to compute the Verschiebung
action on $\omega_1^{(j)}$ and $\omega_2^{(j)}$, for $j=1,\dots,g$. Note that
the Verschiebung takes exact forms to zero, and hence $V(\omega_1^{(j)}) =
-V(\omega_2^{(j)})$; it therefore suffices to compute $V(\omega_1^{(j)})$. This
can be written down explicitly:
\begin{align*}
    V(\omega_1^{(j)}) & =
    V\left(\sum^j_{k=0}(k-2j)a_kx^{k-j-1}\frac{dx}{2y}\right)\\
    & = \sum^j_{k=0}(k-2j)a_k^{1/p}V\left(x^{k-j-1}\frac{dx}{2y}\right).
\end{align*}
At this point, the computation is exactly that done in Lemma \ref{vaction}: we
can write
\begin{align*}
    V(\omega_1^{(j)}) & =
    \sum^j_{k=1}(k-2j)a_k^{1/p}\sum^{g}_{i=1}c^{1/p}_{ip-k+j}\frac{x^i}{2y}dx\\
    & =
    \sum_{i=1}^{g}\sum^j_{k=1}\frac{(k-2j)}{2}a_k^{1/p}c^{1/p}_{ip-k+j}\frac{x^i}{y}dx.
\end{align*}
In the basis
$(\wt{\eta}_1,\cdots,\wt{\eta}_g,\wt{\tau}_1,\cdots,\wt{\tau}_{g})$, the matrix
of Verschiebung is therefore determined by the following facts:
\begin{itemize}
    \item The coefficient of $\wt{\eta}_i$ in $V(\wt{\eta}_j)$ is zero.
    \item The coefficient of $\wt{\tau}_i$ in $V(\eta_j)$ is
	$$\sum^j_{k=0}{(k-2j)}a_k^{1/p}c^{1/p}_{ip-k+j}.$$
    \item The coefficient of $\wt{\eta}_i$ in $V(\wt{\tau}_j)$ is zero.
    \item The coefficient of $\wt{\tau}_i$ in $V(\wt{\tau}_j)$ is given by
	$c^{1/p}_{ip-j}$.
\end{itemize}
We can now graduate to a harder computation.

\subsection{The Frobenius action}
Because $d(x^p) = px^{p-1} dx = 0$, we have
$$F(\wt{\tau}_i) = 0.$$
Similarly,
$$F(\omega_1^{(j)}, \omega_2^{(j)}, y/x^j) = (0, 0, y^p x^{-pj}).$$

\noindent Write
$$f(x)^{\frac{p-1}{2}} = \sum_{k=0}^N c_k x^k,$$
where $N = d(p-1)/2$, so that
$$y^px^{-pj} = y(y^2)^{(p-1)/2}x^{-pj} = y\sum_{k=0}^N c_k x^{k-pj}.$$

\noindent If we let $i = pj-k$, then
$$F(yx^{-j}) = y^px^{-pj} = \sum_{i = pj-N}^{pj} c_{pj-i}y/x^i.$$

\noindent For all $i \in \{1, \ldots, g\}$, we subtract $c_{pj-i} \wt{\eta}_i$
from the sum to we get an element of the form
$$\left(\sum_{i=1}^g -c_{pj-i} \omega_1^{(i)}, \sum_{i=1}^g -c_{pj-i}
\omega_2^{(i)}, \sum_{i \notin \{1, \ldots, g\} } c_{pj-i}y/x^i\right).$$

\noindent Note that \emph{a priori}, we should be looking at $i \geq
\max\{pj-N, 1\}$, but if $pj-N >i$, then $pj-i >N$; so if we set $c_k = 0$ for
$k>N$, we can just subtract $c_{pj-i}\wt{\eta}_i$ for $i = 1, \ldots, g$.

All of the remaining $i$ in the sum are either less than $1$ or greater than
$g$, so we know that $(0, \omega_2^{(i)}, y/x_i)$ is a coboundary for $i \leq
0$, and that $(\wt{\omega}^{(i)}, 0, y/x_i)$ is a coboundary for $i \geq g+1$.
Working in $H^1_{dR}(C)$ --- i.e., modulo coboundaries --- we have:
$$F(\wt{\eta}_j) - \sum_{i=1}^g c_{pi-i}\wt{\eta}_i = \left( \sum_{i=1}^g
-c_{pj-i}\omega_1^{(i)} + \sum_{i = g+1}^{jp} -c_{pj-i}\wt{\omega}^{(i)},
\sum_{i \leq 0} c_{pj-i} \omega_2^{(i)}, 0\right).$$
Since this is an element of $H^1_{dR}(C)$, the difference of the differential
forms in the first and second slot must be equal to differential of the third
slot;
in our case, this implies that the two forms patch up to a globally defined
holomorphic differential form. (Note that this implies that all of the negative
powers of $x$ and the powers greater than $g$ cancel). We may therefore write:
$$\sum_{i=1}^g -c_{pj-i}\omega_1^{(i)} + \sum_{i = g+1}^{jp}
-c_{pj-i}\wt{\omega}^{(i)} = d_1 \frac{dx}{2y} + d_2 \frac{xdx}{2y} + \cdots +
d_g \frac{x^{g-1}dx}{2y}.$$

As both sides of the above equation are holomorphic differential forms, this is
an equality of two functions (and not just an equality in cohomology). Since
all of the $\omega_1^{(i)}$ are defined such that they have no positive powers
of $x$, the only positive powers of $x$ can come from the ${\omega'}^{(i)}$. We
write
\begin{align*}
    \sum_{i = g+1}^{jp} c_{pj-i}\wt{\omega}^{(i)}
    & = \sum_{i = g+1}^{jp} c_{pj-i} \left( \sum_{k=0}^d (k-2i)a_k
    x^{k-i-1}\right)\frac{dx}{2y}\\\
    & = \sum_{i = g+1}^{jp}  \sum_{k^\prime=-i}^{d-i} c_{pj-i}(k'-i)a_{k'+i}
    x^{k'-1}\frac{dx}{2y},
\end{align*}
where we set $k^\prime = k-i$. For a fixed $k^\prime \in \{0, \ldots, g-1\}$,
we therefore have $i \geq \max\{g+1, -k\}$.\footnote{Since $a_{i+k^\prime} = 0$
for $i > d-k^\prime$, the condition that $i\leq d-k^\prime$ is irrelevant.} If
$i \geq (j-1)p$, then $pj-i <p$; so $c_{pj-i} = 0$.\\

\noindent It follows that the coefficient of $x^{k^\prime-1}\ dx/2y$ is
$$-\sum_{i = g+1}^{d-k^\prime} c_{pj-i}a_{k^\prime+i}(k^\prime-i).$$
Thus, in the basis $(\wt{\eta}_1, \ldots, \wt{\eta}_g, \wt{\tau}_1, \ldots,
\wt{\tau}_g)$, the matrix of Frobenius is completely determined by the
following facts:
\begin{itemize}
    \item The $\wt{\eta}_i$ coefficient of $F(\wt{\eta}_j)$ is $c_{pj-i}$.
    \item The $\wt{\tau}_i$ coefficient of $F(\wt{\eta}_j)$ is
	$$\sum_{k = g+1}^{d-i} (k-i)c_{pj-k}a_{k+i}.$$
    \item The $\wt{\eta}_i$ coefficient of $F(\wt{\tau}_j)$ is zero.
    \item The $\wt{\tau}_i$ coefficient of $F(\wt{\tau}_j)$ is zero.
\end{itemize}

%% file: applications.tex
\section{Applications}
Using the computations in the previous section, we can compute the matrix for
the Frobenius and Verschiebung on the Jacobian of any hyperelliptic curve, as
well as the Ekedahl--Oort type. The computations done below were performed with
the use of \texttt{Magma}; the code utilized is available at
\url{https://github.com/sanathdevalapurkar/dieudonne-modules}.
\subsection{A note on the curves $y^2=x^d+1$}

\noindent From our code, we discovered an interesting pattern regarding the
$p-$rank of curves of the form $y^2=x^d+1$ modulo various primes $p$ in the
case $d$ is also a prime. They will always be ordinary in the case that $p = 1
\pmod d$, but if $p \neq 1 \pmod d$ they will often have very low $p-$rank. The
strongest result occurs when both $d$ and $g = \frac{d-1}{2}$ are prime.\\

\begin{introthm1}
    Let $d$ be an odd prime such that $g = \frac{d-1}{2}$ is also prime. Let
    $p$ be a prime not equal to $d$. Then the Jacobian of the hyperelliptic
    curve $y^2=x^d+1$ is ordinary if $p \equiv 1 \pmod d$, and has $p$-rank $0$
    if $p \not\equiv 1 \pmod d$.
\end{introthm1}
\begin{proof}
    Recall that the $p$-rank of a hyperelliptic curve is equal to the limiting
    rank of the Hasse--Witt matrix $H$, defined by $H_{i,j} = c_{pi-j}$ where
    $$\sum_{k} c_k x^k = (x^d+1)^{(p-1)/2}.$$
    Thus we have that
    $$c_{rd} = {(p-1)/2 \choose r },$$
    and $c_k$ is zero for $k$ not a multiple of $d$.\\

    \noindent Suppose $p \equiv 1 \pmod d$, and set $p = md+1$. We have that
    $$H_{i,i} = {(p-1)/2 \choose mi}.$$
    Since $i \leq g = (d-1)/2$ it is easy to check that $(p-1)/2 > mg$, and
    therefore the matrix has nonzero entries on the diagonal. Since $i, j \in
    \{1, \dots, g\}$ and $pi-j \equiv i-j \pmod d$, we can only have $d\mid
    pi-j$ if $i=j$. Thus there can be no other nonzero entries, so $H$ is an
    invertible diagonal matrix and $C$ is ordinary. \\

    \noindent Suppose $p\not\equiv 1\pmod d$. Assume that $p \equiv -1 \pmod
    d$. Then $pi-j \equiv -(i+j) \pmod d$, and since $i$ and $j$ range from $1$
    to $g$, the sum $i+j$ ranges from $2$ to $2g$ and so can never be $0$ mod
    $d$. Thus $c_{pi-j}$ is zero for all entries and the Hasse--Witt matrix is
    identically $0$, so in particular $C$ has $p$-rank $0$. \\

    \noindent The remaining case is when $p$ is neither $1$ nor $-1$ mod $d$.
    Then we claim $H$ will be nilpotent, which is equivalent to having
    characteristic polynomial $P_H(T)=T^g$. We write
    $$P_H(T) = \det(TI - H) = \sum_{\sigma \in S_n} \prod_{i=1}^g (\delta_{i,
    \sigma(i)}T-H_{i, \sigma(i)}).$$

    \noindent Since we know that $pi-i = (p-1)i$ cannot be $0$ mod $d$, we know
    that $H_{i, i} = 0$, and so the coefficient of $T^k$ will be a sum of
    products
    $\prod_{i \notin \mathrm{Fix}(\sigma)} H_{i, \sigma(i)},$
    where $\sigma$ fixes $k$ points and the product now ranges only over the
    elements not fixed by $\sigma$. To prove this equals $T^n$, it suffices to
    show that for any nontrivial cycle $(\alpha_1, \dots, \alpha_k)$, we have
    $$\prod_{i=1}^k H_{\alpha_i, \alpha_{i+1}} = 0.$$
    
    \noindent Consider such a cycle, and suppose the above product is not $0$.
    Then we have that $p\alpha_i - \alpha_{i+1} \equiv 0 \pmod d$, and
    therefore $p\alpha_i \equiv \alpha_{i+1} \pmod d$ for all $i$, and $p
    \alpha_k \equiv \alpha_i \pmod d$. Therefore $p^k \equiv 1 \pmod d$, so $p$
    must have order $k$ in $(\mathbb{Z}/d)^\times$. Since $d$ is prime, this is
    the cyclic group $\mathbb{Z}/2g$, and since $g$ is prime this means $k \in
    \{1, 2, g, 2g\}$. The cycle clearly cannot have length $2g$, and since $p
    \not\equiv 1$ or $-1$ it cannot have order $1$ or $2$, so the only
    important case is $k=g$. If the product $\prod_{i=1}^n H_{i, \sigma(i)}$
    over a cycle of length $g$ was nonzero, it would imply in particular that
    each row of $H$ has a nonzero element.
    
    \noindent But if $p \not\equiv 1 \pmod d$, this is impossible, because
    there will be some $i \in \{1, \dots, g\}$ such that $pi$ mod $d$ lies
    between $g+1$ and $2g$. Then for such an $i$, $pi-j$ will never be $0$ mod
    $d$ for $j \in \{1, \dots, g\}$, so the $i$th row of $H$ will be zero.
    Thus $H$ has characteristic polynomial $T^g$, so is nilpotent, and so $C$
    has $p$-rank $0$.\end{proof}

\subsection{Distribution of final types for low primes}
At low primes, we can count the number of final types coming from hyperelliptic
curves. The data presented below is written in the form
\begin{center}
    \texttt{<final type, number of copies>.}
\end{center}
\begin{lstlisting}
Degree: 3
Prime: 3
The number of final types coming from hyperelliptic curves is 2 (out of 2 
possibilities).

{
    <[ 0 ], 12>,
    <[ 1 ], 24>
}
Time taken: 0.010
-------------------------------------------------------------
Degree: 5
Prime: 3
The number of final types coming from hyperelliptic curves is 3 (out of 4 
possibilities).

{
    <[ 0, 1 ], 36>,
    <[ 1, 2 ], 240>,
    <[ 1, 1 ], 48>
}
Time taken: 0.170
-------------------------------------------------------------
Degree: 7
Prime: 3
The number of final types coming from hyperelliptic curves is 4 (out of 8 
possibilities).

{
    <[ 1, 2, 2 ], 684>,
    <[ 1, 2, 3 ], 1884>,
    <[ 1, 1, 2 ], 240>,
    <[ 0, 1, 2 ], 108>
}
Time taken: 5.400
-------------------------------------------------------------
Degree: 9
Prime: 3
The number of final types coming from hyperelliptic curves is 11 (out of 16 
possibilities).

{
    <[ 1, 1, 2, 3 ], 528>,
    <[ 0, 1, 2, 2 ], 48>,
    <[ 1, 1, 2, 2 ], 240>,
    <[ 1, 2, 2, 3 ], 1920>,
    <[ 1, 2, 3, 3 ], 5232>,
    <[ 0, 1, 2, 3 ], 204>,
    <[ 0, 1, 1, 2 ], 48>,
    <[ 0, 0, 1, 2 ], 12>,
    <[ 1, 2, 2, 2 ], 576>,
    <[ 1, 2, 3, 4 ], 17388>,
    <[ 1, 1, 1, 2 ], 48>
}
Time taken: 363.380
-------------------------------------------------------------
Degree: 3
Prime: 5
The number of final types coming from hyperelliptic curves is 2 (out of 2 
possibilities).

{
    <[ 1 ], 320>,
    <[ 0 ], 80>
}
Time taken: 0.170
-------------------------------------------------------------
Degree: 5
Prime: 5
The number of final types coming from hyperelliptic curves is 4 (out of 4 
possibilities).

{
    <[ 0, 1 ], 240>,
    <[ 1, 1 ], 1760>,
    <[ 1, 2 ], 7920>,
    <[ 0, 0 ], 80>
}
Time taken: 50.920
\end{lstlisting}
Our algorithm is not optimal; one should be able to write a faster algorithm to
gather more data for higher primes and larger genus.
\subsection{A question of Glass-Pries}
Let $G$ be the group scheme corresponding to the $p$-torsion of an abelian
variety of dimension $3$ with $p$-rank $0$ and $a$-number $1$. One can show
that
$$D(G) \simeq k[F,V]/(F^4, V^4, F^3-V^3),$$
where $D(G)$ is the Dieudonn\'e module of $G$. 

In \cite[Question 5.9]{glass-pries}, Glass and Pries ask the following question.
\begin{question}\label{glasspriesqn}
    Are there smooth hyperelliptic curves $C$ and $D$ of genus $4$ and $5$,
    respectively, such that
    $$\Jac(C) \simeq G\oplus(\Z/p\oplus \mu_p),\text{ and }\Jac(C') \simeq
    Q\oplus(\Z/p\oplus\mu_p)^2?$$
\end{question}
We can answer this question in the affirmative for small primes using our Magma
code. A simple linear algebra exercise allows one to compute the matrices for
the Frobenius and Verschiebung. Using our code, this allows us to determine the
canonical type:
\begin{center}
    \begin{tabular}{c|c}
	& Canonical type\\
	\hline
	\hline
	$G\oplus(\Z/p\oplus\mu_p)$
	& $[1, 1, 2, 3]$\\
	\hline
	$G\oplus(\Z/p\oplus\mu_p)^2$
	& $[1,2,2,3,4]$\\
    \end{tabular}
\end{center}
By running through all possible (smooth) hyperelliptic curves, and comparing
the associated canonical types, we can find the desired curves. For genus $4$,
we find that the following hyperelliptic curves answer Question
\ref{glasspriesqn}, thus proving Theorem \ref{thm2}:
\begin{itemize}
    \item At $p=3$: the curve
	$$y^2 = 2x^9 + x^8 + x.$$
    \item At $p=5$: the curve
	$$y^2 = x^9 + x^8 + x.$$
    \item At $p=7$: the curve
	$$y^2 = 5x^9 + 3x^7 + x^6 + x.$$
\end{itemize}
For genus $5$, we find that the following hyperelliptic curves answer Question
\ref{glasspriesqn}:
\begin{itemize}
    \item At $p=3$: the curve
	$$y^2 = 2x^{11} + x^9 + x^7 + x.$$
    \item At $p=5$: the curve
	$$y^2 = 3x^{11} + 2x^{10} + 4x^9 + 2x^8 + x^7 + x.$$
\end{itemize}